\pdfoutput=1 %

\documentclass[12pt]{amsart}
     \usepackage{amsmath,amsthm,amssymb, amsfonts,mathtools,physics}
     \usepackage[a4paper,top=3cm,bottom=3cm,left=2cm,right=2cm]{geometry} %
     \usepackage{setspace}
     \usepackage{enumitem}
     \usepackage{comment}
     \usepackage{graphicx}
     \usepackage{pdfpages}
     \usepackage{tikz}
     \usepackage{xcolor}           
     \allowdisplaybreaks

     \usepackage[          
          pdfsubject={Mathematics},
          colorlinks=true, linkcolor=blue, linkbordercolor=blue, citecolor=red, citebordercolor=red, urlcolor=blue, linktocpage=true
     ]{hyperref}
     \usepackage{aliascnt}
     \usepackage[nameinlink,noabbrev]{cleveref}
     \numberwithin{equation}{section}
     \theoremstyle{definition}
     \newtheorem{definition}{Definition}[section]
     \newaliascnt{example}{definition}
     
     \aliascntresetthe{example}
     \newaliascnt{problem}{definition}
     
     \aliascntresetthe{problem}
     \theoremstyle{definition}
     \newaliascnt{theorem}{definition}
     \newtheorem{theorem}[theorem]{Theorem}
     \aliascntresetthe{theorem}
     \newaliascnt{lemma}{definition}
     \newtheorem{lemma}[lemma]{Lemma}
     \aliascntresetthe{lemma}
     \newaliascnt{corollary}{definition}
     \newtheorem{corollary}[corollary]{Corollary}
     \aliascntresetthe{corollary}

     \crefname{definition}{Definition}{Definitions}
     \crefname{example}{Example}{Examples}
     \crefname{problem}{Problem}{Problems}
     \crefname{theorem}{Theorem}{Theorems}
     \crefname{lemma}{Lemma}{Lemmas}
     \crefname{corollary}{Corollary}{Corollaries}

     \newcommand{\C}{\mathbb{C}} %
     \newcommand{\geh}{\mathcal{G}}

     \newcommand{\injten}{\otimes_{\varepsilon}}
     \newcommand{\projten}{\otimes_{\pi}}
      
     \newcommand{\inv}{^{-1}}

     \newcommand{\act}{\curvearrowright} 
     \newcommand{\tca}{\curvearrowleft}
     \newcommand{\Forall}{\text{ for any }}

     \DeclareMathOperator{\Ad}{Ad}
     \DeclareMathOperator{\id}{id}
     \newcommand{\kaigyo}{\textcolor{white}{.}}  %

\title[Amenability of group actions and Banach algebras]{Amenability of group actions on compact spaces and the associated Banach algebras}
\author{Hikaru Awazu}
\makeatletter
\let\@cite@ofmt\@firstofone
\renewenvironment{thebibliography}[1]{%
  \section*{\refname}\markboth{\MakeUppercase{\refname}}{\MakeUppercase{\refname}}%
  \normalfont\normalsize
  \list{\@biblabel{\theenumiv}}{%
    \settowidth\labelwidth{\@biblabel{#1}}%
    \setlength\labelsep{10pt}%
    \setlength\leftmargin{\labelwidth}\addtolength\leftmargin{\labelsep}%
    \setlength\itemsep{0pt}\setlength\parsep{0pt}%
    \usecounter{enumiv}}%
  \interlinepenalty5000\widowpenalty10000\clubpenalty10000
  \raggedbottom\frenchspacing
  \Urlmuskip=0mu plus 3mu\relax
  \mathchardef\UrlBigBreakPenalty=100\relax
  \mathchardef\UrlBreakPenalty=200\relax
}{\endlist}
\makeatother
\usepackage{embedfile}
\embedfile[mimetype=application/x-tex,desc={LaTeX source}]{ba.tex}

\begin{document}

\renewcommand{\thefootnote}{\fnsymbol{footnote}}
\footnote[0]{The University of Tokyo, Graduate School of Mathematical Sciences}
\renewcommand{\thefootnote}{\arabic{footnote}}

\begin{abstract}
     For a topological group $G$, amenability can be characterized by the amenability of the convolution Banach algebra $L^1(G)$. 
     Here a Banach algebra $A$ is called amenable 
     if every bounded derivation from $A$ into any dual--type $A$--$A$--Banach bimodule is inner. 

     We extend this classical result to the case of discrete group actions on compact Hausdorff spaces 
     in our main theorem \autoref{thm:main}. 
     By using a Banach algebra naturally associated with the action 
     and adopting a suitably weakened notion of amenability for Banach algebras, 
     we obtain an analogous characterization of amenable actions. 

     As a lemma, we also proved a fixed--point property 
     for amenable actions in \autoref{thm:Awazu fixed point} that strengthens the theorem of Dong and Wang (2015). 
\end{abstract}

\maketitle 

\tableofcontents

\section{Introduction and basic definitions}\label{sec:introduction_and_basic_definitions}

\subsection{Notations} 

     \hspace{5mm}

     \begin{itemize}

     \item In this article, let $X$ be a nonempty compact Hausdorff space. 

     We write $C(X)$ for the $C^{*}$--algebra of $\mathbb{C}$--valued continuous functions on $X$, equipped with the supremum norm $\|\cdot\|_\infty$. 

     Throughout, we work over $\C$ as the scalar field for vector spaces. 

     \item Let $\Gamma$ be a discrete group, and suppose that $\Gamma$ acts continuously on $X$.

     We denote the action by $g.x$ for $g \in \Gamma$ and $x\in X$. 

     The symbol $e\in\Gamma$ denotes the unit of $\Gamma$. 

     \item Then $\Gamma$ acts isometrically on $C(X)$. We write $p^g$ for the image of $p\in C(X)$ under $g\in \Gamma$. 

     Here $p^g$ is defined by $p^g(x) = p(g^{-1}.x)$ \; .

     \item $\mathrm{Prob}(\Gamma)$ is the convex subset of $\ell ^1 (\Gamma)$ consisting of positive elements of norm one. 

     It has the natural left action of $\Gamma$ defined by 
     \[
     g.f(h) := f(g^{-1}h) \quad \text{for } g,h \in \Gamma,\: f\in \mathrm{Prob}(\Gamma) 
     \]

     and is equipped with $\ell ^1$--norm $\|\cdot\|_1$. 

     \item Topological groups are always assumed to be second--countable, Hausdorff, and locally compact.   

     \end{itemize}

\subsection{Amenable groups and amenable actions}
     \hspace{5mm}   	

     First we review (topological) amenable groups: 

     \begin{definition}\label{def:amen groups}

     Let $G$ be a topological group with its Haar measure $\mu$. 

     We say that $G$ is \textit{amenable} if there exists $\phi \in L^{\infty} (G,\mu) ^{*}$ 
     which is positive, unital, and left $G$--invariant, i.e., 
     \[
     \phi (g.f) = \phi(f)  \text{ for } g\in G, \: f \in L^{\infty} (G,\mu) \; .
     \] 

     We call this $\phi$ a \textit{left invariant mean} on $G$. 

     \end{definition}

     Amenability of groups has many characterizations and is widely used for analyzing groups by operator algebraic techniques. 
     Among these, one can easily find the following using $\ell^{\infty} (\Gamma)^{*} = \ell^1 (\Gamma) ^{**}$: 

     \begin{theorem}\label{thm:approximate mean of group}

     The group $\Gamma$ is amenable if and only if it has a norm--$1$ net $(f_i)_i$ in $\ell^1 (\Gamma)$ 
     satisfying 
     \[
     \Forall g \in \Gamma, \quad \|f_i \ast \delta_g -f_i \|_1 \xrightarrow{\hspace{5mm} i\hspace{5mm}} 0 \; .
     \]

     \end{theorem}

     This $(f_i)_i$ is called a \textit{right approximate mean} for $\Gamma$.

     For group actions, we use the following formulation of amenability~\cite[Definition~2.1]{AnantharamanDelaroche02}.

     \begin{definition}\label{def:amenable action}
     We say that the topological action $\Gamma \act X$ is \textit{amenable} 
     if there exists a net $( m_i )_i$ in $C(X,\mathrm{Prob}(\Gamma))$ satisfying 

     \[
     \Forall g \in \Gamma,\quad \sup\limits_{x\in X} \| m_i(g.x) - g.(m_i(x)) \|_1 \xrightarrow{\hspace{5mm} i\hspace{5mm}} 0 \; .
     \]

     \end{definition}

     Amenability of group actions is also used in operator--theoretic research on (discrete) groups. 
     There is a large class of groups called \textit{exact groups}.
     Finding a non-exact group was a difficult problem until Gromov constructed one \cite{Gromo03}.

     In 2000, Ozawa showed that a discrete group $\Gamma$ is exact 
     if and only if the action of $\Gamma$ on its Stone--\v{C}ech compactification $\beta \Gamma$ is amenable \cite{Ozawa00}.  
     Therefore, amenable actions are particularly used for analysing exact groups. 

     \hspace{5mm}

\subsection{Amenable Banach algebras and Johnson's theorem}
     \kaigyo 

     First, we introduce the definition of amenable Banach algebras.

     \begin{definition}\label{def:amenable Banachalg}
      For a Banach algebra $A$, we make definitions as follows. 

     \begin{enumerate}

     \item We say that a Banach space $E$ is an $A$--$A$--\textit{Banach bimodule} when $E$ has commuting left and right contractive actions of $A$.  

     \item Then $E^{*}$ also has the structure of an $A$--$A$--Banach bimodule by letting 

     \[
     a.\phi .b (v) \: := \: \phi (b.v.a)
     \]
     for $a,b \in A$, $v\in E$, $\phi \in E^{*}$. 

     \item We say that $D:A \rightarrow E$ is a \textit{derivation} if $D$ is bounded and linear, and 

     \[
     D(ab) = a.D(b) + D(a).b 
     \]
     for all $a,b \in A$. 

     \item We say the derivation $D$ is \textit{inner} if there exists $v \in E$ such that 
     \[
     D(a) = a.v - v.a \; .
     \] 
     The right side is denoted by $\Ad_v (a)$. 

     \item We say that $A$ is an \textit{amenable Banach algebra} 
     if for any $A$--$A$--bimodule $E$ and for any derivation $D:A \rightarrow E^{*}$, $D$ is inner. 

     \end{enumerate}

     If $A$ is unital and $E$ is a unital $A$-bimodule, then every derivation $D:A\to E$ satisfies $D(1_A)=0$. 

     \end{definition}

     Johnson proved the connection among amenable groups, amenable Banach algebras and vanishing of bounded cohomologies of (discrete) groups. 

     The bounded cohomology $ \{ H_b ^n (\Gamma ;V) \}_n $ of $\Gamma$ 
     with coefficients in a normed $\mathbb{C}[\Gamma]$-module $V$ on which $\Gamma$ acts isometrically
     is a variation of the ordinary group cohomology.  
     This is obtained by restricting cochains of the group cohomology to uniformly bounded ones 
     with respect to the norm $\|\cdot\|_V$.  
     For the definition, see~\mbox{\cite[\(\S\)\,1.3]{Frige16}}. 

     \begin{theorem}[Johnson]\label{thm:Johnson} (\cite{Johns72}, \cite[Theorem~2.1.10]{Runde20}, \cite[\(\S\)\,3.4]{Frige16})
     
     For a topological group $G$, the following statements are equivalent: 

     \begin{enumerate} 

     \item The group $G$ is amenable. 

     \item The Banach algebra $L^1(G,\mu)$ equipped with the convolutional product $\ast$ is amenable. 

     \item[] Moreover, when $G$ is discrete, the following are also equivalent:

     \item $ H_b ^1(G; (\ell ^{\infty}(G) / \mathbb{C})^*) = 0 $

     \item For every normed $\mathbb{C}[G]$-module $V$ with isometric $G$-action
     and every $n\ge 1$, we have $H_b ^n (G; V^{*}) = 0$ \; .

     \end{enumerate}

     \end{theorem}

\subsection{Proof of Johnson's Theorem and Fixed--Point Theorem for Amenable Groups}\label{sub:proof_of_johnson_s_theorem_and_fixed_point_theorem_for_amenable_groups}
     \kaigyo 

     We will give a sketch of the proof of \autoref{thm:Johnson} 
     because we obtained our main theorems by imitating this proof and it makes understanding our proof clear. 
     We work for the case that $G$ is discrete (and use the symbol $\Gamma$) for brevity.  

     To show (2) $\Rightarrow$ (1) in \autoref{thm:Johnson}, 
     it suffices to construct a concrete derivation from $\ell ^1 (\Gamma)$ 
     and use the fact that every derivation is inner. 

     \begin{itemize}

     \item We define the Banach space $E$ by 
     \[E := \ell ^{\infty} (\Gamma) /\mathbb{C} 1_\Gamma\]

     where $1_\Gamma \in  \ell ^{\infty} (\Gamma)$ is the constant--$1$ function.

     \item Then, 
     \[
     E^{*} \: \cong \: \{ \tau \in \ell ^{\infty} (\Gamma)^{*} \: | \: \tau (1_\Gamma) = 0 \} \; .
     \] 

     \item Equip $E=\ell^\infty(\Gamma)/\mathbb C1_\Gamma$ with the bimodule structure induced by
     \[
          f.\phi=\Bigl(\sum_{g\in\Gamma}f_g\Bigr)\phi,
          \qquad
          (\phi.f)(h)=\sum_{g\in\Gamma}f_g\phi(gh).
     \]
     These bounded actions commute and preserve $\mathbb C1_\Gamma$. The corresponding dual actions on $E^*$ are
     \[
          (f.\tau)([\phi])=\tau([\phi.f]),
          \qquad
          \tau.f=\Bigl(\sum_{g\in\Gamma}f_g\Bigr)\tau.
     \]

     \item Also, $\ell ^{\infty} (\Gamma)^{*} $ is an $\ell ^1(\Gamma)$--bimodule in a similar manner to the case of $E^{*}$. 

     \vspace{2mm}

     \item Fix $\tau _0 \in \ell ^{\infty} (\Gamma)^{*}$ such that $\tau_0(1_\Gamma) = 1$. 

     \vspace{2mm}

     \item Since
     \[
          (f.\tau_0-\tau_0.f)(1_\Gamma)
          =\Bigl(\sum_{g\in\Gamma}f_g\Bigr)\tau_0(1_\Gamma)
          -\Bigl(\sum_{g\in\Gamma}f_g\Bigr)\tau_0(1_\Gamma)=0,
     \]
     the bounded derivation $D(f):=f.\tau_0-\tau_0.f$ takes values in $E^*$.

     \item Using amenability of $\ell ^1 (\Gamma)$, we obtain $\tau _1 \in E^{*}$ with 
     \[
     f.\tau _0 - \tau _0 .f = f. \tau _1 - \tau _1 .f \; .
     \] 

     \item Put $m_0:=\tau_0-\tau_1$. Then $|m_0|$ is also left invariant,
     and $m:=|m_0|/|m_0|(1_\Gamma)$ is a left-invariant mean on $\ell^\infty(\Gamma)$.

     \end{itemize}

     \hspace{5mm}

     To show (1) $\Rightarrow$ (2) in \autoref{thm:Johnson}, 
     we invoke Day's Fixed--point characterization of amenable groups~\cite{Day61}, \cite[Theorem~1.5.1]{Runde20}. 

     \begin{theorem}[Day]\label{thm:Day}  For a locally compact group $G$, the following are equivalent: 

     \begin{enumerate}

     \item The group $G$ is amenable. 

     \item For any locally convex space $V$ and any nonempty compact convex subset $K$, 
     if $G$ acts affinely and separately continuously on $K$, then $K$ has a $G$--fixed point. 

     \end{enumerate}

     Here the meaning of an affine action and separate continuity is as follows: 

     \begin{itemize}

     \item Acting \textit{affinely} on $K$ means 
     $g.(tx + (1-t)y) \: = \: t g.x \: + (1-t) g.y$ is satisfied for all $x,y \in K$ and $t \in [0,1]$.

     \item Acting \textit{separately continuously} on $K$ means 
     $G \times K \rightarrow K$ ;$(g,k) \mapsto g.k$ is separately continuous. 

     \end{itemize}
     \end{theorem}

     \hspace{5mm}

     Since $\ell^1(\Gamma)$ is unital, it suffices to consider unital Banach bimodules.
     Let $E$ be such a bimodule and let $D:\ell^1(\Gamma)\to E^*$ be a bounded derivation. Define

     \begin{align}
          \alpha _g (\tau) \: := \: \delta _g . \tau . \delta _{g ^{-1}} \: - D(\delta _g) . \delta _{g ^{-1}} \label{def:alpha}
     \end{align}

     The derivation identity shows that $\alpha$ is an affine action,
     and each $\alpha_g$ is weak* continuous.

     Then, $\alpha_g(\tau) = \tau$ implies 
     \[
     D(\delta _g) = \delta _g . \tau \, - \, \tau . \delta _g \; .
     \]

     Since $\ell ^1 (\Gamma)$ is the closed linear span of $\{ \delta _g \} _{g \in \Gamma}$, 
     we have $D \, = \, \Ad_{\tau}$ on all of $\ell ^1(\Gamma)$ which shows $D$ is inner. 

     Therefore it suffices to find a fixed point of the action $\alpha$ 
     and find a weak*--compact convex $\alpha$--invariant set $K \subset E^{*}$ to exploit \autoref{thm:Day}. 
     Set 
     \[
     K := \overline{conv}^{wk*} \{ -D(\delta _g). \delta _{g ^{-1}} \: | \: g \in \Gamma \}  \; .
     \] 

     As the weak* closed convex hull of the bounded orbit $\{\alpha_g(0):g\in\Gamma\}$,
     $K$ is nonempty, weak* compact and $\alpha$-invariant.
     Day's theorem therefore gives a fixed point.

\section{Banach Spaces Arising from Topological Group Actions}\label{sec:banach_spaces_arising_from_topological_group_actions}

     In this section, we briefly review the previous results of Monod \cite{Monod10} and Brodzki et al. \cite{Brodz10}, 
     which tell us how to characterize the amenability of group actions in terms of their invariant means and bounded cohomology.

     First, we set basic definitions concerning Banach spaces which are compatible with the given topological action. 

     \begin{definition}\label{def:action module}
      For a group action $\Gamma \act X$, a Banach space $V$ is said to be 
     a \textit{Banach $\Gamma$--$C(X)$--module} if it satisfies the following conditions:  

     \begin{itemize}

     \item The Banach space $V$ admits a left $\Gamma $--action by linear isometries. 

     \item The Banach space $V$ admits a left $C(X)$--action that is contractive. 

     \item Compatibility of actions: $g.(p.(g^{-1}.v)) = p^g .v $ for all $v \in V$, $g \in \Gamma$, and $p \in C(X)$. 

     \end{itemize}

     In this situation, $V^{*}$ has the natural Banach $\Gamma$--$C(X)$--module structure with 
     \begin{itemize}

     \item $g.v^{*} (v) := v^{*} (g^{-1}.v)$ for $g \in \Gamma$, $v\in V$, $v^{*} \in V^{*}$. 

     \item $p.v^{*} (v) := v^{*} (p.v)$ for $p \in C(X)$, $v\in V$, $v^{*} \in V^{*}$. 

     \end{itemize} 
     \end{definition} 

     We recall the Banach spaces used to define invariant means of actions, 
     following Monod~\cite[\(\S\)\,2.3]{Monod10} and Brodzki et al.~\cite{Brodz10}.
     For precise explanations about unconditional summability and injective tensor products, 
     see Ryan~\cite[Chapter~3, especially Example~3.4]{Ryan02}. 

     \begin{definition}\label{def:unconditionally summable}
     Let $V$ be a Banach space.
     A family $(v_i)_{i\in I}$ is \textit{unconditionally summable} if its finite partial sums converge in norm,
     with the finite subsets of $I$ ordered by inclusion.
     \end{definition}

     For a sequence $(v_n)$, this is equivalent to convergence of $\sum\limits_{n\ge1}a_nv_n$
     for every bounded scalar sequence $(a_n)$.

     \vspace{2mm}

     \begin{definition}\label{def:A_0}
      For a group action $\Gamma \act X$,
     \begin{enumerate} 

     \item Put $\mathcal{A}(\Gamma\act X):=C(X,\ell^1(\Gamma))$, and write $\mathcal{A}$ for short.
     For $f\in\mathcal{A}$, write $f_g(x)=f(x)(g)$, and identify $f$ with the family $(f_g)_{g\in\Gamma}$ in $C(X)$.
     Under this identification, $\mathcal{A}$ consists precisely of the unconditionally summable families in $C(X)$ indexed by $\Gamma$.

     \item Equip $\mathcal{A}$ with the norm
     \[
     \|f\|_{\mathcal{A}} := \sup_{x\in X}\|f(x)\|_1
       = \left\|\sum_{g\in\Gamma}|f_g|\right\|_\infty.
     \]
     The series $\sum\limits_{g\in\Gamma}|f_g|$ converges uniformly on $X$.

     For a family $(f_g)_{g\in\Gamma}$ in $C(X)$, the condition
     $\sup_{x\in X}\sum\limits_g|f_g(x)|<\infty$ alone does not imply membership in $\mathcal{A}$.
     Membership is equivalent to the uniform tail condition
     \[
     \inf_{\substack{F\subset\Gamma\\ F\text{ finite}}}
     \sup_{x\in X}\sum_{g\notin F}|f_g(x)|=0.
     \]

     \item The space $\mathcal{A}(\Gamma\act X)$ is complete, and its finitely supported elements form a dense subspace.

     Note that $\mathcal{A}(\Gamma\act X) \cong \ell^1 (\Gamma) \widehat{\otimes}_{\epsilon} C(X)$ where $\widehat{\otimes}_{\epsilon}$ is the completed injective tensor product of Banach spaces. 

     \item The map $\bar{\pi} : \mathcal{A} \rightarrow C(X)$ is defined by 
     \[
     \bar{\pi} (f) :=  \sum_{g\in \Gamma} f_g 
     \] 
     The series converges in $C(X)$, and $\bar{\pi}$ is a bounded linear map.

     \item The Banach space $\mathcal{A}(\Gamma\act X)$ admits a left $\Gamma$--action with 
     \[
     (g.f)_h := (f_{g^{-1}h})^g \quad \text{for } g,h \in \Gamma, \: f\in \mathcal{A}
     \] 
     which is linear and isometric.  

     \item The Banach space $\mathcal{A}(\Gamma\act X)$ admits a $C(X)$--action with 
     \[
     (p.f)_h := p \cdot f_{h} \quad \text{for } p \in C(X),\: f \in \mathcal{A},\: h \in \Gamma 
     ,\] 
     where the product on the right is the pointwise product in $C(X)$. 

     \item These actions satisfy $g.(p.f)=p^g.(g.f)$, so $\mathcal{A}$ is a Banach $\Gamma$--$C(X)$--module. 

     \end{enumerate}
     \end{definition}

     \begin{definition}\label{def:W_0}
     For a group action $\Gamma \act X$, define
     \[
     \begin{aligned}
     W_0(\Gamma,X)
     &:=\overline{\{f\in\mathcal{A}: f\text{ is finitely supported},\
          \bar\pi(f)\in\mathbb{C}1_X\}}^{\,\|\cdot\|_{\mathcal{A}}}\\
     &=\{f\in\mathcal{A}:\bar\pi(f)\in\mathbb{C}1_X\}.
     \end{aligned}
     \]
     The inclusion from left to right follows from the continuity of $\bar\pi$.
     For the reverse inclusion, given $\bar\pi(f)=c1_X$, approximate $f$ by a finitely supported $u\in\mathcal{A}$
     and add $c1_X-\bar\pi(u)$ to its $e$-component; the resulting element $v$ has sum $c1_X$ and satisfies
     \[
     \|v-f\|_{\mathcal{A}}\le 2\|u-f\|_{\mathcal{A}}.
     \]
     Define $\pi\in W_0^*$ by $\bar\pi(f)=\pi(f)1_X$.
     \end{definition}

     Note that $W_0$ is not a Banach $\Gamma$--$C(X)$--submodule of $\mathcal{A}$ since it is not closed under the $C(X)$--action. 
     By contrast, $\ker\pi$ is a Banach $\Gamma$--$C(X)$--submodule of $\mathcal{A}$.

\subsection{Characterizations of amenable actions with bounded cohomology}
     The bounded cohomological characterization in \autoref{thm:Nowak2} provides background for \autoref{def:ell^1}, although it is not used in the main results of the paper.
 
     \kaigyo 

     Now amenable actions can be formulated using invariant means: 
     \begin{theorem}\label{thm:Nowak1} (See \cite[\(\S\)\,1]{Monod10} and, for countable groups, \cite[arXiv version, Theorem~A]{Brodz10}.)  

     For a group action $\Gamma \act X$, the following are equivalent: 
     \begin{enumerate}

     \item The action $\Gamma \act X$ is amenable. 

     \item There exists $\mu \in W_0(\Gamma,X)^{**}$ such that 
     $\mu (\pi) = 1$ and $\mu$ is $\Gamma $--invariant with the $\Gamma $--action defined on $\mathcal{A}(\Gamma\act X)^{**}$.  

     \end{enumerate}

     This $\mu$ is called an \textit{invariant mean for $\Gamma \act X$}. 

     \end{theorem}

     \hspace{2mm}

     We concisely introduce the characterization of amenable actions using bounded cohomology, 
     since the module condition there is similar to that of our results. 
     As a preliminary, we define some useful properties of $C(X)$--actions.

     \begin{definition}\label{def:ell^1} For a $C(X)$--module Banach space $V$, we introduce the following definitions. 
     \begin{enumerate}

     \item The action $C(X) \act V$ is called \textit{$\ell ^{\infty}$--geometric} or \textit{type(C)} if 

     \[
          \left\| \sum_{1\le k \le n} p_k.v_k \right\|_V \: \le \: 
          \left\| \sum_{1\le k\le n} p_k \right\|_{\infty} \cdot \max_{1 \le k \le n} \| v_k \|_V 
     \]

     \hspace{5mm}

      for all $\bigl\{p_k\bigr\}_{k=1}^n\subset C(X,[0,1])$ and $\bigl\{v_k\bigr\}_{k=1}^n\subset V$. 

     \hspace{5mm}

     \item The action $C(X) \act V$ is called \textit{$\ell ^1$--geometric} or \textit{type(M)} if 

     \[
          \sum_{1\le k \le n} \| p_k.v \|_V \: \le \:   \left\| \sum_{1\le k\le n} p_k \right\|_{\infty} \cdot \| v \|_V 
     \]

     \hspace{5mm}

      for all $\bigl\{p_k\bigr\}_{k=1}^n \subset C(X,[0,1])$ and $v \in V$.

     \end{enumerate}

     \end{definition}

     \hspace{5mm}
     Regarding these properties, the following can be easily proved: 

     \begin{lemma}\label{lem:Nowak dual} (For (1) and (2), see \cite[arXiv version, Lemma~6]{Brodz10}.)
     \begin{enumerate}

     \item If $C(X) \act V$ is $\ell ^1$--geometric, then $C(X) \act V^{*}$ is $\ell^{\infty}$--geometric. 

     \item If $C(X) \act V$ is $\ell ^{\infty}$--geometric, then $C(X) \act V^{*}$ is $\ell^1$--geometric. 

     \item The $C(X)$--modules $\mathcal{A}$, $\ker\pi$, and double--dual of these have $\ell^{\infty}$-geometric $C(X)$-actions. 

     \end{enumerate}
     \end{lemma}

     $H_b^n(\Gamma,E)$ denotes the bounded cohomology of $\Gamma$ with coefficients in $E$; see \cite{Brodz10}.

     The following form of the cohomological characterization of Brodzki et al.~\cite[arXiv version, Theorem~B]{Brodz10} and Monod~\cite[\(\S\)\,1]{Monod10} extends \autoref{thm:Johnson}.

     \begin{theorem}\label{thm:Nowak2}
     For a group action $\Gamma \act X$, the following are equivalent.
     \begin{enumerate}
     \item The action $\Gamma \act X$ is amenable.

     \item We have $H_b^1\bigl(\Gamma,(\ker\pi)^{**}\bigr)=0$.

     \item We have $H_b^n(\Gamma,V^*)=0$ for all $n\geq 1$ and any
     $\Gamma$--$C(X)$-module $V$ with a unital $\ell^1$-geometric $C(X)$-action.
     \end{enumerate}
     \end{theorem}

     \begin{proof}
     For $(1)\Rightarrow(3)$, apply Monod's characterization~\cite[\(\S\)\,1]{Monod10} to the underlying real module of $V$.
     The map $\lambda\mapsto\operatorname{Re}\lambda$ identifies the underlying real space of $V^*$ with the real dual of $V$.

     For $(3)\Rightarrow(2)$, take $n=1$ and $V=(\ker\pi)^*$, which is $\ell^1$-geometric by \autoref{lem:Nowak dual}.

     Finally, assume (2). The bounded $(\ker\pi)^{**}$-valued cocycle $g\mapsto\delta_g-\delta_e$ is a coboundary, so there exists $v\in(\ker\pi)^{**}$ such that $\delta_g-\delta_e=g.v-v$ for every $g\in\Gamma$.
     Thus $m=\delta_e-v\in W_0(\Gamma,X)^{**}$ is $\Gamma$-invariant and satisfies $m(\pi)=1$.
     Its real part is invariant and sums to $1_X$, so Monod's invariant-mean criterion gives (1).
     \end{proof}

     \hspace{5mm}

\subsection{Algebraic structure of $\mathcal{A}(\Gamma\act X)$}
     \kaigyo 

     Monod pointed out that $\mathcal{A}(\Gamma\act X)$ has a Banach algebra structure defined below: 

     \begin{definition}\label{def:convolution} (\cite[\(\S\)\,2.3]{Monod10})

     For $f_1,f_2\in\mathcal{A}(\Gamma\act X)$, define

     \[
          (f_1 * f_2)(g) \: := \: \sum\limits_{h\in \Gamma} f_1 (h) \cdot (f_2(h^{-1}g))^h \; .
     \]

     \end{definition}

     For completeness, we prove closure under convolution and the norm estimate.

     \begin{lemma}\label{lem:convolution}
     For $f_1,f_2\in\mathcal{A}$, we have $f_1*f_2\in\mathcal{A}$ and
     \[
          \|f_1*f_2\|_{\mathcal{A}}
          \le \|f_1\|_{\mathcal{A}}\|f_2\|_{\mathcal{A}}.
     \]
     Thus $\mathcal{A}$ is a Banach algebra.
     \end{lemma}

     \begin{proof}
     First suppose that $f_1,f_2$ have finite support in $\Gamma$.
     Then $f_1*f_2$ also has finite support and continuous coordinates, so it belongs to $\mathcal{A}$.
     Moreover,
     \begin{align*}
     \|f_1 \ast f_2 \|_{\mathcal{A}}
     &= \sup_{x\in X} \left( \sum_{g\in \Gamma} \left| \sum_{h\in \Gamma} f_1( h,x) f_2(h^{-1}g,h^{-1}.x) \right| \right) \\
     &\le   \sup_x \left( \sum_h |f_1( h,x) | \cdot \sum_g \left| f_2(h^{-1}g,h^{-1}.x)\right| \right) \\
     &=     \sup_x \left( \sum_h |f_1( h,x) | \cdot \sum_g \left| f_2(g,h^{-1}.x)\right| \right) \\
     &\le   \sup_x \left( \left(\sum_h |f_1( h,x) | \right) \cdot
     \sup_{h'} \left( \sum_g \left| f_2(g,h'^{-1}.x)\right| \right) \right) \\
     &\le \left(\sup_x \sum_h |f_1 (h,x)| \right) \cdot \left( \sup_{x',h'} \sum_g |f_2(g, h'^{-1}.x') | \right) \\
     &= \|f_1\|_{\mathcal{A}} \|f_2\|_{\mathcal{A}}.
     \end{align*}
     By density, this product extends uniquely to a bounded bilinear map
     $\mathcal{A}\times\mathcal{A}\to\mathcal{A}$ satisfying the same estimate.
     The defining series converges in $C(X)$ for each $g$, since
     \[
     \sup_{x\in X}\sum_{h\notin F}
     |f_1(h,x)f_2(h^{-1}g,h^{-1}.x)|
     \le \|f_2\|_{\mathcal{A}}
     \sup_{x\in X}\sum_{h\notin F}|f_1(h,x)|
     \longrightarrow 0
     \]
     as $F$ increases over the finite subsets of $\Gamma$;
     finite-support approximation therefore identifies it with the extended product.
     Associativity holds on finitely supported elements and hence on $\mathcal{A}$ by continuity.
     \end{proof}

     We remark that the norm of this Banach algebra $\mathcal{A}(\Gamma\act X)$ coincides with
     the range-fibre norm $\|\cdot\|_{I,r}$, which is one of the two terms in Renault's $I$-norm
     on $C_c(\geh)$ for the transformation groupoid $\geh=\Gamma\ltimes X$.
     Indeed, $\|f\|_I=\max\{\|f\|_{\mathcal{A}},\|f^*\|_{\mathcal{A}}\}$ (Remark~1.38 of \cite{Willi19}).
     The completion of $C_c(\geh)$ with $\|\cdot\|_{\mathcal{A}}$ also forms a Banach algebra.
     It is written as $\mathcal{E}$ in Section~9.6 of \cite{Willi19}.

     On $C_c(\Gamma,C(X))$, this product is the usual crossed-product convolution. 
     However, the involution need not be isometric with respect to $\|\cdot\|_{\mathcal{A}}$,
     so $\mathcal{A}(\Gamma\act X)$ need not be a $B^*$-algebra with this involution.

     \vspace{2mm}

     For $h \in \Gamma$, let $\delta _h \in \mathcal{A}(\Gamma\act X)$ denote the element defined by $g \mapsto \delta(g,h) \cdot 1_X$ 
     and denote $\delta _e$ by $1$. 
     Note that $\delta_g \ast f = g.f$ and $(f \ast \delta_g)_h= f_{hg^{-1}}$ for all $f \in \mathcal{A}$. 
     In particular, $\delta_g \ast \delta_h = \delta_{gh}$, 
     and $\ell^1(\Gamma)$ is a Banach subalgebra of $\mathcal{A}(\Gamma\act X)$. 

     Meanwhile, $C(X)$ is also a Banach subalgebra of $\mathcal{A}(\Gamma\act X)$. 
     For $p \in C(X)$, we use the same symbol $p \in \mathcal{A}$ to denote the map $g \mapsto \delta(e,g) \cdot p$. 
     Then $p \ast f = p.f$, and $\delta_g \ast p \ast \delta_{g\inv} = p^g$.
     In the tensor-product description, the two embeddings are $a\mapsto a\otimes 1_X$
     and $p\mapsto\delta_e\otimes p$, respectively. 

     Moreover, $\mathcal{A}(\Gamma\act X)$ is generated as a Banach algebra by $\{ \delta_g \}_{g \in \Gamma}$ and $\{ p  \}_{p\in C(X)}$.

\section[{Amenability of $W_0(\Gamma,X)$}]{Amenability of \texorpdfstring{$W_0(\Gamma,X)$}{$W_0(\Gamma,X)$}}\label{sec:amenability_of_w_0_gamma_x_}

     When working with a group $G$ without actions, invariant means should live in $L^1(G,\mu)^{**}$, 
     and amenability of $G$ is characterized by that of $L^1(G,\mu)$. 
     Then it is natural that amenability of an action $\Gamma \act X$ can be characterized by that of $W_0(\Gamma, X)$, 
     whose double--dual is the space in which invariant means may reside. 

     However, N. Ozawa pointed out to me that amenability of $W_0(\Gamma, X)$ is too strong a condition:

     \begin{theorem}\label{thm:Ozawa} For an action $\Gamma \act X$, 
     if $W_0(\Gamma, X)$ is amenable as a Banach algebra, then 
     $\Gamma$ is amenable as a discrete group.
     \end{theorem}

     To prove this, we need some lemmas on amenable Banach algebras: 

     \begin{lemma}\label{lem:approximate unit} (\cite[Corollary~2.3.10 and Proposition~2.2.1, respectively]{Runde20}) 

     \begin{enumerate}

     \item Let $A$ be an amenable Banach algebra, 
     and let $I \le A$ be a closed ideal of finite codimension; 
     then $I$ is also amenable. 

     \item Let $A$ be a (non--unital) amenable Banach algebra. 

     Then $A$ has a \textit{bounded approximate unit} $( e_i )_i \subset A$, 
     i.e., $\sup\limits_i \| e_i\| < \infty$ and for all $a\in A$: 
     \[
     \| e_i a - a \| \xrightarrow{\hspace{2mm}i\hspace{2mm}} 0 , \quad  \| a e_i -a \| \xrightarrow{\hspace{2mm}i\hspace{2mm}} 0
     \]

     \end{enumerate}
     \end{lemma}

     \begin{proof}[Proof of \autoref{thm:Ozawa}]

     First, it follows from \autoref{lem:approximate unit} (1) that $\ker\pi$ is amenable as a Banach algebra, 
     since $\ker\pi$ has codimension one in $W_0(\Gamma,X)$. 

     Then, using \autoref{lem:approximate unit} (2) for $\ker\pi$ 
     we obtain its bounded approximate unit $(e_i)_i \subset \ker\pi$. 
     Fix $x\in X$ and put $f_i(g)=\delta(g,e)-e_i(g,x)$ for $g\in\Gamma$.
     Since $\pi(e_i)=0$, we have $\sum\limits_{g\in\Gamma}f_i(g)=1$ and hence $\|f_i\|_1\geq 1$.
     For $h\in\Gamma$, evaluation at $x$ gives
     \[
          \|f_i\ast\delta_h-f_i\|_1
          \leq \|(\delta_e-e_i)\ast(\delta_h-\delta_e)\|_{\mathcal A}
          \longrightarrow 0,
     \]
     since $\delta_h-\delta_e\in\ker\pi$ and $(e_i)_i$ is an approximate unit for $\ker\pi$.
     Applying \autoref{thm:approximate mean of group} to the norm-one net
     $(f_i/\|f_i\|_1)_i$ shows that $\Gamma$ is amenable.
     \end{proof}

\section{Fixed--point Characterizations of Amenable Actions}\label{sec:fixed_point_characterizations_of_amenable_actions}

     In contrast to the previous section, part of the proof of Johnson's theorem still remains valid for amenable actions. 
     First, we should formalize the fixed--point characterization of amenable actions. 

     In 2015, Dong and Wang proved the fixed--point theorem for amenable actions with respect to isometric linear actions on Banach spaces~\cite[Theorem~2.7, stated over the reals]{Dong15}. 
     As a preliminary, we introduce an analogue of convex sets in the context of $C(X)$--Banach modules. 

     \begin{definition}\label{def:continuous-functions-convex} For a $\Gamma$--$C(X)$--module $V$, 
     a subset $K \subset V$ is called \textit{$C(X)$--convex} if 

     \[
          \sum\limits_{1\le k \le n} p_k.c_k \; \in \; K
     \]
     for all $\bigl\{p_k\bigr\}_{k=1}^n  \subset C(X,[0,1])$ with $\sum\limits_{k=1}^n p_k = 1_X$ 
     and $\bigl\{c_k\bigr\}_{k=1}^n \subset K$. 

     \hspace{5mm}

     \end{definition}

     We call such a family $(p_k)_{k=1}^n$ a \textit{finite decomposition of $1_X$}
     and the sum a \textit{$C(X)$--convex combination}.

     \begin{theorem}[Dong and Wang]\label{thm:DongWang} \: Suppose that $\Gamma$ is finitely generated. Then the following are equivalent. 
     \begin{enumerate}
     \item The action $\Gamma \act X$ is amenable. 

     \item For any $\ell ^1$--geometric $\Gamma$--$C(X)$--module $V$ 
     and for any $C(X)$--convex, weak*-compact, and nonempty subset $K \subset V^*$ 
     with $\Gamma .K \subset K$, 
     $K$ has a $\Gamma$--fixed point. 

     \end{enumerate}
     \end{theorem}

     We now return to an arbitrary discrete group $\Gamma$.
     We slightly generalize this theorem to apply it to the affine action $\alpha$ of \autoref{def:alpha}. 
     As a preliminary, we recall the structure of affine actions of groups \cite[\(\S\)\,2.1]{FernosValette12}.

     Let $\alpha: \Gamma \act V$ be an affine action on a linear space $V$. 
     Then there exist a linear action $\hat{\alpha}: \Gamma \act V$ and a cocycle map $c:\Gamma \rightarrow V$ such that 

     \[
       \alpha_g (v) = \hat{\alpha}_g (v) + c_g  \; .
     \]
     Here the cocycle satisfies $c_{gh} = c_g + \hat{\alpha}_g (c_h)$. 
     We write $\alpha = (\hat{\alpha},c)$ for this decomposition.

     For a locally convex space $V$, we introduce the topological vector space $L_{\sigma}(V)$, 
     which is the set of all linear maps on $V$ endowed with the topology of pointwise convergence in~$V$~\mbox{\cite[\(\S\)\,46]{Munkres00}}. 

     \begin{definition}\label{def:group-continuous-functions-locally-convex} For $\Gamma \act X$, 
     a Hausdorff locally convex space $V$ is called a \textit{$\Gamma$--$C(X)$ locally convex module} 
     if it is equipped with the following: 

     \begin{itemize}
     \item The group $\Gamma$ acts affinely on $V$ by $\alpha=(\hat{\alpha},c)$,
     and each $\alpha_g\colon V\to V$ is continuous with linear part $\hat{\alpha}_g$.

     \item There is a continuous linear unital map $\beta: C(X) \rightarrow L_{\sigma}(V)$ 
     which satisfies $\hat{\alpha}_g \circ \beta(p) \circ \hat{\alpha}_{g^{-1}} = \beta(p^g)$.
     Continuity of $\beta$ means that $p\mapsto\beta(p)(v)$ is continuous for each $v\in V$.  

     \end{itemize}

     \end{definition}

     We write $\beta (p,v)$ instead of $\beta (p)(v)$ for $p \in C(X)$, $v\in V$.

     We define $C(X)$--convexity as in \autoref{def:continuous-functions-convex},
     with $\beta(p,v)$ in place of $p.v$.

     In this setting, we prove the generalized version of \autoref{thm:DongWang}: 

     \begin{theorem}\label{thm:Awazu fixed point} Let $\Gamma \act X$ be an amenable action and  
     $(V; \alpha = (\hat{\alpha},c) ;\beta)$ be a $\Gamma$--$C(X)$ locally convex module with respect to $\Gamma \act X$. 

     If a subset $K \subset V$ is $\alpha$--invariant, $C(X)$--convex, compact, and non--empty, 
     then $K$ has an $\alpha$--fixed point.   

     \end{theorem}

     \begin{definition}\label{def:W0-positive-cone}
     Let $W_0^+=\{f\in W_0 : f_g\ge 0 \text{ for every }g\in\Gamma\}$.
     \end{definition}

     For every finitely supported $f\in W_0$, there exist finitely supported
     $f_0,\ldots,f_3\in W_0^+$ such that $f=\sum\limits_{k=0}^{3}i^k f_k$ and
     \begin{equation}\label{eq:four-term-estimate}
       \sum\limits_{k=0}^{3}\pi(f_k)
       =\|\operatorname{Re}f\|_{\mathcal A}+\|\operatorname{Im}f\|_{\mathcal A}
       \le 2\|f\|_{\mathcal A}.
     \end{equation}

     Suppose first that $f$ is real-valued, and let $f^+$ and $f^-$ denote its
     positive and negative parts. Set
     \[
       p:=\|\bar\pi(f^+)\|_\infty 1_X-\bar\pi(f^+)\in C(X)^+,\qquad
       f_0:=f^++p\delta_e,\quad f_2:=f^-+p\delta_e.
     \]
     Then $f=f_0-f_2$, and both $f_0$ and $f_2$ are nonnegative and finitely supported.
     Since $\bar\pi(f^-)=\bar\pi(f^+)-\pi(f)1_X$, both have constant sums and belong
     to $W_0^+$. Moreover,
     \[
       \pi(f_0)+\pi(f_2)=2\|\bar\pi(f^+)\|_\infty-\pi(f)=\|f\|_{\mathcal A}.
     \]
     Applying this construction to the real and imaginary parts of $f$ gives
     \eqref{eq:four-term-estimate}.
     The uniform constant $2$ in this four-term decomposition is optimal.

\begin{proof}[Proof of \autoref{thm:Awazu fixed point}]
     \kaigyo

     \textbf{Step 1.} \quad Find a fixed point.

     We may take the approximately invariant net $(f_i)$ given by amenability
     to consist of finitely supported elements of $W_0^+$.
     For each $m$ in the original net, choose a finite set $F\subset\Gamma$ containing $e$.
     Add $\sum\limits_{h\notin F}m_h$ to $m_e$ and set all coefficients outside $F$ to zero.
     This preserves the sum $1_X$ and gives arbitrarily close approximations to
     $m$ in $\|\cdot\|_{\mathcal A}$ as $F$ increases. Thus
     \[
       \pi(f_i)=\|f_i\|_{\mathcal A}=1,\qquad
       \|g.f_i-f_i\|_{\mathcal A}\longrightarrow 0\quad(g\in\Gamma).
     \]
     Let $\mu$ be a weak* cluster point of $(f_i)$ in the unit ball of $W_0^{**}$.
     Then $\mu(\pi)=1$ and $g.\mu=\mu$ for every $g\in\Gamma$.

     Fix $c_0\in K$. For each $i$, the nonzero coefficients $f_i(h)$ form
     a finite decomposition of $1_X$. Since $\alpha_h(c_0)\in K$ for
     every $h\in\Gamma$, $C(X)$--convexity gives
     \[
       \sum\limits_{h\in\Gamma}\beta\bigl(f_i(h),\alpha_h(c_0)\bigr)\in K.
     \]

     \textbf{Step 2.} \quad Define the functional $\psi^c$.

     Fix $c\in K$ and $\psi\in V^*$. For each finitely supported $f\in W_0$, set
     \[
       \psi^c(f):=\psi\!\left(\sum\limits_{h\in\Gamma}\beta(f_h,\alpha_h(c))\right).
     \]
     This is linear in $f$ since $\beta$ is linear in its first variable.

     Choose $f_0,\ldots,f_3$ as in \eqref{eq:four-term-estimate}.
     For each nonzero $f_k$, the family
     $\left\{\dfrac{f_k(h)}{\pi(f_k)}\right\}_{h\in\operatorname{supp}f_k}$
     is a finite decomposition of $1_X$. Since $\alpha_h(c)\in K$, the corresponding
     $C(X)$--convex combination belongs to $K$. Thus, for each nonzero $f_k$,
     \[
       |\psi^c(f_k)|=\pi(f_k)\left|\psi\!\left(
       \sum\limits_{h\in\operatorname{supp}f_k}
       \beta\!\left(\frac{f_k(h)}{\pi(f_k)},\alpha_h(c)\right)
       \right)\right|\le \pi(f_k)\max_{v\in K}|\psi(v)|.
     \]
     Using \eqref{eq:four-term-estimate}, we obtain
     \[
     \begin{aligned}
       |\psi^c(f)|&\le \sum\limits_{k=0}^{3}|\psi^c(f_k)|\\
       &\le \left(\sum\limits_{k=0}^{3}\pi(f_k)\right)\max_{v\in K}|\psi(v)|
       \le 2\|f\|_{\mathcal A}\max_{v\in K}|\psi(v)|.
     \end{aligned}
     \]
     By density, $\psi^c$ extends to a bounded linear functional on $W_0$.

     For $g\in\Gamma$ and finitely supported $f\in W_0$, the covariance of $\beta$
     and the identity $\hat{\alpha}_g(\alpha_h(c))=\alpha_{gh}(c)-c_g$ give
     \begin{equation}\label{eq:affine-average}
     \begin{aligned}
       &\alpha_g\!\left(\sum\limits_{h\in\Gamma}\beta(f_h,\alpha_h(c))\right)\\[5pt]
       &\quad=\sum\limits_{h\in\Gamma}
         \hat{\alpha}_g\!\left(\beta(f_h,\alpha_h(c))\right)+c_g\\[5pt]
       &\quad=\sum\limits_{h\in\Gamma}
         \beta\bigl((f_h)^g,\alpha_{gh}(c)-c_g\bigr)+c_g\\[5pt]
       &\quad=\sum\limits_{h\in\Gamma}\beta\bigl((f_h)^g,\alpha_{gh}(c)\bigr)
         -\sum\limits_{h\in\Gamma}\beta\bigl((f_h)^g,c_g\bigr)+c_g\\[5pt]
       &\quad=\sum\limits_{k\in\Gamma}\beta\bigl((g.f)_k,\alpha_k(c)\bigr)
         -\beta\!\left(\sum\limits_{h\in\Gamma}(f_h)^g,c_g\right)+c_g\\[5pt]
       &\quad=\sum\limits_{k\in\Gamma}\beta\bigl((g.f)_k,\alpha_k(c)\bigr)
         +(1-\pi(f))c_g .
     \end{aligned}
     \end{equation}
     In the last equality we used $\sum\limits_{h\in\Gamma}(f_h)^g=\pi(f)1_X$
     and $\beta(1_X)=\id_V$. For $f=f_i$, the last term vanishes since $\pi(f_i)=1$.

     \textbf{Step 3.} \quad Show that $\tilde c$ is fixed.

     On a subnet along which $f_i\to\mu$ weak*, take a further subnet on which the
     averages from Step~1 converge in $K$, and denote their limit by $\tilde c$.
     For every $g\in\Gamma$ and $\psi\in V^*$, continuity of $\alpha_g$ and
     \eqref{eq:affine-average} give, along this subnet,
     \[
     \begin{aligned}
       \psi(\alpha_g(\tilde c))
       &=\lim_i\psi^{c_0}(g.f_i)\\
       &=(g.\mu)(\psi^{c_0})=\mu(\psi^{c_0})=\psi(\tilde c).
     \end{aligned}
     \]
     Since $V^*$ separates points of $V$, we have $\alpha_g(\tilde c)=\tilde c$
     for every $g\in\Gamma$.
\end{proof}

\section{Johnson's Theorem for Topological Actions}\label{sec:johnson_s_theorem_for_topological_actions}

     As we noted in \autoref{sec:amenability_of_w_0_gamma_x_}, 
     we should weaken the amenability of Banach algebras to characterize the amenability of actions: 

     \begin{definition}\label{def:right-continuous-functions-ell1} For a Banach algebra $A$ that includes $C(X)$ as a Banach subalgebra, 
     \begin{enumerate}
     \item 
     We say that an $A$--$A$--bimodule $E$ is \textit{right--$C(X)$--$\ell^1$--geometric} if 
     the right action $E \tca C(X)$ obtained by restricting the action of $A$ is $\ell ^1$--geometric. 

     \item
     We say that $A$ is \textit{right--$C(X)$--$\ell^1$--amenable} if 
     for any $A$--$A$--bimodule $E$ that is right--$C(X)$--$\ell^1$--geometric, and 
     for any bounded derivation $D: A \rightarrow E^*$, $D$ is inner. 

     \end{enumerate}
     \end{definition}

     And we show the following: 

     \begin{theorem}\label{thm:main} For $\Gamma \act X$, the following are equivalent:

     \begin{enumerate}

     \item The action $\Gamma \act X$ is amenable. 

     \item The Banach algebra $\mathcal{A}(\Gamma\act X)$ is right--$C(X)$--$\ell^1$--amenable. 

     \end{enumerate}
     \end{theorem}

     In this section, we prove only (1)$\Rightarrow$(2). 

     \begin{proof}[Proof of \autoref{thm:main} (1)$\Rightarrow$(2)] 

     \hspace{5mm}

     Take an $\mathcal{A}$--$\mathcal{A}$--bimodule $E$ that is right--$C(X)$--$\ell^1$--geometric, 
     and a bounded derivation $D: \mathcal{A} \longrightarrow E^*$. 

     \hspace{5mm}

     \textbf{Step 1.} \quad Show that we may assume that $D$ is left $C(X)$--equivariant (i.e., $D(p\ast a) = p.D(a)$). 

     Consider the restriction of $D$ to $C(X)$. 
     Then $D|_{C(X)} : C(X) \longrightarrow _{C(X)} E^* _{C(X)} $ is a derivation again. 

     Since $C(X)$ is a commutative $C^*$-algebra and thus an amenable Banach algebra, $D|_{C(X)}$ is inner. 
     That is, there exists $\tau_0 \in E^*$ such that 
     $ D(p) = p.\tau_0 - \tau_0 .p $ for all $p \in C(X)$.  

     \vspace{2mm}

     Then $D - \Ad_{\tau_0}$ is a left $C(X)$--equivariant bounded derivation. Indeed, 
     \[
          (D - \Ad_{\tau_0}) (p \ast a) = (D - \Ad_{\tau_0})(p).a + p.(D - \Ad_{\tau_0})(a) = p.(D - \Ad_{\tau_0})(a) \; .
     \]

     If we show the theorem for left $C(X)$--equivariant derivations, 
     then we obtain $\tau$ with 
     \[
     D - \Ad_{\tau_0} = \Ad_{\tau} \; .
     \] 
     This shows $D = \Ad_{\tau + \tau_0} $ and $D$ is inner. 
     We remark that this technique can be used for any amenable subalgebras.

     Replacing $D$ by $D-\Ad_{\tau_0}$, we have $D(1)=0$ and hence
     $D(a)=1.D(a)=D(a).1$ for every $a\in\mathcal A$, so we may work with
     the unital bimodule $1.E.1$, whose dual is identified with $1.E^*.1$. 

     \hspace{2mm}

     \textbf{Step 2.} Use the fixed point theorem \autoref{thm:Awazu fixed point}.

     As in \eqref{def:alpha}, we define an affine action $\alpha: \Gamma \act E^* $ by 

     \[
          \alpha _g (\tau) \: := \: \delta _g . \tau . \delta _{g ^{-1}} \: - D(\delta _g) . \delta _{g ^{-1}} \; . 
     \]

     Set $\beta : C(X) \rightarrow L_{\sigma}(E^*)$ as $\beta(p)(\tau) := p.\tau$ for $\tau \in E^*$ and $p \in C(X)$. 

     Each $\alpha_g$ is continuous, and $\beta$ is unital, linear, and continuous. Moreover, 
     \[ 
     \delta_g .(p.(\delta _{g\inv} . \tau . \delta_g)). \delta_{g\inv } = p^g .\tau 
     \]
     shows the compatibility between $(\alpha,\beta)$ and $\Gamma \act X$. 

     Therefore $(E^*, \alpha, \beta)$ is a $\Gamma$--$C(X)$ locally convex module. 

     Define $K^{\circ} \subset E^*$ to be 
     \[
          K^{\circ} := C(X)\text{--convex combinations of }\{ - D(\delta _g) . \delta _{g ^{-1}} \mid g \in \Gamma \}. 
     \]

     Now, the norm of each $c_g := - D(\delta _g) . \delta _{g ^{-1}} \in E^*$ is at most $\|D\|$. 
     Then we can show that 
     the assumption that $E$ is right--$C(X)$--$\ell^1$--geometric ensures that $K^{\circ}$ is norm--bounded by $\|D\|$.

     By \autoref{lem:Nowak dual} (1), $C(X) \act E^*$ is ''left--$C(X)$--$\ell^{\infty}$--geometric'', i.e., 

     \[
          \left\| \sum_{1\le k \le n} p_k.\tau_k \right\|_{E^*} \: \le \: \left\| \sum_{1\le k \le n} p_k \right\|_{\infty} \cdot \sup_{1\le k \le n} \|\tau_k\|_{E^*}
     \]

     and thus $K^{\circ}$ is norm--bounded by $\|D\|$. 

     Therefore, $K:= \overline{K^{\circ}} ^{wk*}$ is also bounded and weak*--compact by the Banach--Alaoglu theorem. 
     Since $c_h=\alpha_h(0)$, \eqref{eq:affine-average} gives
     \[
       \alpha_g\!\left(\sum\limits_{j=1}^{n}p_j.c_{h_j}\right)
       =\sum\limits_{j=1}^{n}p_j^g.c_{gh_j}
     \]
     for every finite decomposition $(p_j)_{j=1}^{n}$ of $1_X$ and all
     $g,h_1,\ldots,h_n\in\Gamma$. Thus $K^\circ$ is $\alpha$--invariant.
     Since each $\alpha_g$ and each $\beta(p)$ is weak* continuous,
     $K$ is $\alpha$--invariant and $C(X)$--convex.

     \vspace{2mm}

     Since the action $\Gamma\act X$ is amenable, \autoref{thm:Awazu fixed point}
     gives a fixed point $\tau_*\in K$ satisfying
     \[
     D(\delta_g) = \delta_g.\tau_* - \tau_* .\delta_g
     \] 
     for all $g \in \Gamma$. 
     Since $D$ is left $C(X)$--equivariant, 
     \[
          D(p \ast \delta_g) = p.D(\delta_g) = (p \ast \delta_g).\tau_* - p.\tau_* .\delta_g
          \quad \text{for all }p \in C(X), \: g\in \Gamma .
     \]
     \hspace{5mm}

     \textbf{Step 3.} Show that $\tau_*$ is $C(X)$--central (i.e., $p.\tau_* = \tau_* .p$) and that $D = \Ad_{\tau_*}$. 

     First, we can confirm that $c_g := -D(\delta_g).\delta_{g^{-1}}$ is $C(X)$--central 
     since $D$ is a $C(X)$--equivariant derivation: 

     \begin{align*}
     p.D(\delta_g).\delta_{g^{-1}} 
     &= D(p\ast\delta_g).\delta_{g^{-1}} \\
     &= D(p\ast \delta_g \ast \delta_{g^{-1}}) - (p\ast \delta_g).D(\delta_{g^{-1}}) \\
     &= D(p) - (\delta_g \ast p^{g^{-1}}). D(\delta_{g^{-1}}) \\
     &= p.D(1) - \delta_g . D(p^{g^{-1}}\ast \delta_{g^{-1}})  \\
     &= 0 - \delta_g . D(\delta_{g^{-1}} \ast p) \\
     &= - D(p) + D(\delta_g).(\delta_{g^{-1}} \ast p) \\
     &= 0 + (D(\delta_g).\delta_{g^{-1}}).p \\
     \end{align*}

     The $C(X)$--centrality is preserved under taking $C(X)$--convex combinations since $C(X)$ is commutative. 
     Therefore every element of $K^{\circ}$ is $C(X)$--central.
     The $C(X)$-central elements form a weak* closed subspace of $E^*$, so $\tau_*$ is $C(X)$-central.

     By Steps 2 and 3, we obtain

     \[
          D(p \ast \delta_g) = (p \ast \delta_g).\tau_* - \tau_* .(p\ast\delta_g)
          \quad \text{for all }p \in C(X), \: g\in \Gamma .
     \]

     Since $\mathcal{A}(\Gamma\act X)$ is the closed linear span of $\{ p \ast \delta_g \mid p\in C(X), \: g\in \Gamma \}$, 
     it follows that $D = \Ad_{\tau_*}$.  

     \end{proof}

     \noindent\textit{Remark.}
     For the real-valued differences $g.f_i-f_i$ in the fixed-point argument, the constant $2$
     in \eqref{eq:four-term-estimate} can be replaced by $1$.

\section{Proof of the reverse implication}\label{sec:reverse-implication}

As in the group case, we prove $(2)\Rightarrow(1)$ of \autoref{thm:main}
by applying the innerness hypothesis to a specific bounded derivation.
An earlier approach, due to N.~Ozawa (personal communication), used the
amenability of measured groupoids. The derivation below gives an invariant
mean in $W_0(\Gamma,X)^{**}$ directly. The author proposed using an
operator-valued derivation for this implication.

\subsection{Preliminaries}

\begin{definition}\label{def:reverse-action-on-functions}
Let $\Gamma$ act on the compact Hausdorff space $X$, and write
$\mathcal{A}=C(X,\ell^1(\Gamma))$. For $a=(a_g)_{g\in\Gamma}\in\mathcal{A}$
and $p\in C(X)$, put
\[
    a.p=\sum_{g\in\Gamma} a_g p^g,
    \qquad p^g(x)=p(g^{-1}.x).
\]
\end{definition}

The series converges uniformly since $\sum\limits_g |a_g|$ converges uniformly
on $X$ and $|p^g|\leq\|p\|_\infty$. In particular,
\[
    \|a.p\|_\infty\leq\|a\|_{\mathcal{A}}\|p\|_\infty .
\]
For finitely supported $a,b\in\mathcal{A}$, the identity $(p^h)^g=p^{gh}$
gives
\[
    a.(b.p)=\sum_{g,h\in\Gamma}a_g b_h^g p^{gh}=(a*b).p .
\]
By density, this identity holds for all $a,b\in\mathcal{A}$. Thus $C(X)$
is a contractive left $\mathcal{A}$-module. In particular,
\[
   \delta_g.p=p^g,\qquad (q\delta_e).p=qp
   \qquad\bigl(g\in\Gamma,\ q\in C(X)\bigr).
\]

\begin{lemma}\label{lem:reverse-augmentation}
The map $\bar\pi:\mathcal{A}\to C(X)$ is a left $\mathcal{A}$-module
map. Its kernel is the closed left ideal $\ker\pi$. The embedding
\[
   J:C(X)\longrightarrow\mathcal{A},\qquad p\longmapsto p\delta_e
\]
is an isometric linear section of $\bar\pi$.
\end{lemma}

\begin{proof}
For $a,b\in\mathcal{A}$, we have $\bar\pi(b)=b.1_X$. Hence
\[
   \bar\pi(a*b)=(a*b).1_X=a.(b.1_X)=a.\bar\pi(b).
\]
Since $W_0(\Gamma,X)=\bar\pi^{-1}(\mathbb{C}1_X)$ and
$\bar\pi(f)=\pi(f)1_X$ on $W_0(\Gamma,X)$, we have $\ker\bar\pi=\ker\pi$.
The displayed identity shows that this kernel is a left ideal, and
continuity of $\bar\pi$ gives closedness. Finally, $\bar\pi(p\delta_e)=p$
and $\|p\delta_e\|_{\mathcal{A}}=\|p\|_\infty$.
\end{proof}

We write $a.b=a*b$ for the left regular action of $\mathcal{A}$, including
its restriction to $\ker\pi$. We denote the completed projective tensor
product over $\mathbb C$ by $\otimes_\pi$.

\begin{lemma}\label{lem:reverse-tensor-actions}
Let $A$ be a Banach algebra, $V$ a left Banach $A$-module, and $W$ a
right Banach $A$-module. Then $V\otimes_\pi W$ is a Banach $A$-bimodule
with actions
\[
   a.(v\otimes w)=(a.v)\otimes w,\qquad
   (v\otimes w).a=v\otimes(w.a).
\]
Its dual identifies isometrically with $B(V,W^*)$ under the pairing
\[
   \langle T,v\otimes w\rangle=\langle T(v),w\rangle .
\]
Under this identification, the right $A$-action on $V\otimes_\pi W$ induces
the left $A$-action on $B(V,W^*)$ by duality, and the left action on the
tensor product likewise induces the right action on $B(V,W^*)$. Explicitly,
\[
   (a.T)(v)=a.T(v),\qquad (T.a)(v)=T(a.v),
\]
where the left action on $W^*$ is defined by
$\langle a.\lambda,w\rangle=\langle\lambda,w.a\rangle$.
\end{lemma}

\begin{proof}
The two actions on elementary tensors are bounded for the projective norm
and commute, so they extend to bimodule actions on the completion. The
isometric identification follows from the universal property of the
projective tensor product and the correspondence between bounded bilinear
forms on $V\times W$ and bounded operators from $V$ to $W^*$
\cite[Section~2.2]{Ryan02}. The order of the two actions follows by
pairing with an elementary tensor.
\[
\begin{aligned}
   \langle(a.T)(v),w\rangle
      &=\langle T,(v\otimes w).a\rangle=\langle T(v),w.a\rangle,\\
   \langle(T.a)(v),w\rangle
      &=\langle T,a.(v\otimes w)\rangle=\langle T(a.v),w\rangle .
\end{aligned}
\]
This proves the asserted formulas.
\end{proof}

\subsection{Proof of \texorpdfstring{$(2)\Rightarrow(1)$}{(2) implies (1)}}

Assume that $\mathcal A$ is right-$C(X)$-$\ell^1$-amenable in the sense
of \autoref{def:right-continuous-functions-ell1}.

\begin{enumerate}
\item \textit{The coefficient bimodule.} Set
\[
   E=C(X)\otimes_\pi(\ker\pi)^* .
\]
The first factor carries the left $\mathcal A$-action of
\autoref{def:reverse-action-on-functions}. The right action on the second factor is
dual to the left action on $\ker\pi$, so that
\[
   \langle\varphi.a,n\rangle=\langle\varphi,a.n\rangle
   \qquad\bigl(\varphi\in(\ker\pi)^*,\ n\in\ker\pi\bigr).
\]
By \autoref{lem:reverse-tensor-actions}, these actions make $E$ an
$\mathcal A$-bimodule, and
\[
   E^*=B\bigl(C(X),(\ker\pi)^{**}\bigr),\qquad
   (a.T)(p)=a.T(p),\quad (T.a)(p)=T(a.p).
\]

The left $C(X)$-action on $\ker\pi$ is pointwise multiplication and is
$\ell^\infty$-geometric by \autoref{lem:Nowak dual}(3). Its dual right
action on $(\ker\pi)^*$ is therefore $\ell^1$-geometric by
\autoref{lem:Nowak dual}(2). For $q\in C(X)$, the restricted right action
on $E$ is
\[
   (p\otimes\varphi).q=p\otimes(\varphi.q).
\]
The isometric flip $p\otimes\varphi\mapsto\varphi\otimes p$ puts this
action on the first factor. Since $C(X)$ is commutative,
\autoref{lem:tensor and geometricity}(1) below applies to that factor
and shows that $E$ is right-$C(X)$-$\ell^1$-geometric.

\item \textit{The derivation.} Given $a\in\mathcal A$ and $p\in C(X)$,
we may first lift $p$ to $p\delta_e$ and then act by $a$, or first form
$a.p$ and then lift it. By \autoref{lem:reverse-augmentation}, both
elements have image $a.p$ under $\bar\pi$. Their difference therefore
belongs to $\ker\pi$. Define
\[
   D(a)(p)=a.(p\delta_e)-(a.p)\delta_e ,
\]
viewing $\ker\pi$ inside $(\ker\pi)^{**}$ by its canonical embedding.
This embedding intertwines the left $\mathcal A$-actions. The formula is
bilinear in $a,p$, and
\[
   \|D(a)(p)\|\leq 2\|a\|_{\mathcal A}\|p\|_\infty .
\]
Thus $D:\mathcal A\to E^*$ is a bounded linear map. For
$a,b\in\mathcal A$ and $p\in C(X)$, the bimodule formulas give
\[
\begin{aligned}
   (a.D(b)+D(a).b)(p)
   &=a.\bigl[b.(p\delta_e)-(b.p)\delta_e\bigr]\\
   &\quad +a.\bigl[(b.p)\delta_e\bigr]-[a.(b.p)]\delta_e\\
   &=(a*b).(p\delta_e)-[(a*b).p]\delta_e\\
   &=D(a*b)(p).
\end{aligned}
\]
Hence $D$ is a derivation.

\item \textit{The invariant mean.} By hypothesis, there exists $S\in E^*$
such that $D(a)=a.S-S.a$ for every $a\in\mathcal A$. Fix $g\in\Gamma$.
Taking $a=\delta_g$ and evaluating at $1_X$, we obtain
\[
   \delta_g-\delta_e=\delta_g.S(1_X)-S(1_X)
   \qquad\text{in }(\ker\pi)^{**},
\]
since $\delta_g.1_X=1_X$.

Let $i:\ker\pi\hookrightarrow W_0(\Gamma,X)$ be the inclusion. Left
multiplication by $\delta_g$ is the given action of $g$ on both spaces,
so $i^{**}$ is a $\Gamma$-equivariant isometric embedding. Identifying
$\delta_e$ with its canonical image in $W_0(\Gamma,X)^{**}$, put
\[
   m=\delta_e-i^{**}\bigl(S(1_X)\bigr).
\]
For each $g\in\Gamma$, the preceding identity gives
\[
   g.m-m=\delta_g-\delta_e-i^{**}\bigl[\delta_g.S(1_X)-S(1_X)\bigr]=0 .
\]
Moreover, $i^*\pi=0$ because $\pi$ vanishes on $\ker\pi$. Consequently,
\[
   m(\pi)=\pi(\delta_e)-\langle S(1_X),i^*\pi\rangle=1 .
\]
Thus $m$ is an invariant mean for $\Gamma\act X$. \autoref{thm:Nowak1}
now implies that the action is amenable. \hfill\(\square\)
\end{enumerate}

\begin{lemma}\label{lem:tensor and geometricity}

     Let $C(X) \curvearrowright V$ be a $C(X)$--Banach module and let $W$ be a Banach space.  

     The projective tensor $V\projten W$ and injective tensor $V\injten W$ are also $C(X)$--Banach modules with the action given by  

     \[
          p.(v\otimes w) := (p.v) \otimes w 
     \]
     for all $v\in V$, $w\in W$, $p\in C(X)$. 
     Then we obtained the following: 

     \begin{enumerate}

     \item If $V$ is $\ell^1$--geometric, then the projective tensor product $V\projten W$ is also $\ell^1$--geometric. 

     \item If $V$ is $\ell^{\infty}$--geometric, then the injective tensor product $V\injten W$ is also $\ell^{\infty}$--geometric. 

     \end{enumerate}
     \end{lemma}

     \begin{proof} 
     \vspace{1mm}
     (1) Let $p_1,\ldots,p_n\in C(X,[0,1])$. For any representation
     $z=\sum\limits_{i=1}^{m}v_i\otimes w_i$ of an algebraic tensor,
     \autoref{def:ell^1}~(2) gives
     \[
     \begin{aligned}
          \sum\limits_{k=1}^{n}\|p_k.z\|_\pi
          &\leq \sum\limits_{i=1}^{m}
               \left(\sum\limits_{k=1}^{n}\|p_k.v_i\|_V\right)\|w_i\|_W\\
          &\leq \left\|\sum\limits_{k=1}^{n}p_k\right\|_\infty
               \sum\limits_{i=1}^{m}\|v_i\|_V\|w_i\|_W.
     \end{aligned}
     \]
     Taking the infimum over all such representations gives the required
     estimate on $V\otimes W$. It extends to $V\projten W$ by density.

     \vspace{1mm}

     (2) Let $p_1,\ldots,p_n\in C(X,[0,1])$ and
     $z_1,\ldots,z_n\in V\injten W$. For $w^*\in W^*$ with $\|w^*\|\leq 1$,
     put $v_k=(\operatorname{id}_V\otimes w^*)z_k$.
     Then $\|v_k\|_V\leq\|z_k\|_\varepsilon$, and \autoref{def:ell^1}~(1) gives
     \[
     \begin{aligned}
          \left\|(\operatorname{id}_V\otimes w^*)
               \left(\sum\limits_{k=1}^{n}p_k.z_k\right)\right\|_V
          &=\left\|\sum\limits_{k=1}^{n}p_k.v_k\right\|_V\\
          &\leq\left\|\sum\limits_{k=1}^{n}p_k\right\|_\infty
               \max_{1\leq k\leq n}\|z_k\|_\varepsilon .
     \end{aligned}
     \]
     Taking the supremum over such $w^*$ proves (2).

     \end{proof}

\section{Character--Amenability} 

     There is a weaker notion of amenability of Banach algebras, called left/right--$\omega$--amenability, 
     where $\omega$ is a character on the Banach algebra (see Section 4.3 of \cite{Runde20}). 

     \begin{definition} 
     Let $A$ be a Banach algebra and $\omega:A\to \C$ be a character (Banach algebra homomorphism) on $A$. 
     Here we permit $0$ as a character. 

     We call $A$ \textit{left (right) $\omega$--amenable} if it satisfies the following: 

     Consider any $A$--$A$--bimodule $E$ whose left (resp. right) action is by 
     \[
     a.v := \omega(a)v \quad (\text{resp.} v.a := \omega(a)v ) 
     \] for $a\in A$, $v\in E$. 
     Then every bounded derivation $D:A\to E^*$ is inner. 
     \end{definition} 

     In particular for $\omega = 0$, we obtain the following immediately: 

     \begin{theorem} \label{thm:0--amenable} (\cite[Proposition~4.3.4(ii)]{Runde20})
     For a Banach algebra $A$, the following are equivalent: 
     \begin{enumerate}

     \item $A$ is right $0$--amenable. 

     \item $A$ has a bounded left approximate identity $(a_i)_i$. 

     \end{enumerate}
     \end{theorem}

     By \cite[Theorem~4.3.9]{Runde20}, the following are equivalent for a locally compact group $G$.
     \begin{enumerate}

     \item The group $G$ is amenable.

     \item The Banach algebra $L^1(G)$ is amenable. 

     \item The Banach algebra $L^1(G)$ is left or right $\omega$--amenable for some nonzero character $\omega$.

     \item The Banach algebra $L^1(G)$ is left (right)--$\omega$--amenable for any character $\omega$. 

     \item The Banach algebra $L^1(G)$ is left (right)--$1_G$--amenable. 

     \end{enumerate}

     For $\Gamma \act X$, we have the following partial analogue: 

     \begin{theorem} \label{thm:0--amenable of N0} 

     For $\Gamma \act X$, the following are equivalent: 
     \begin{enumerate}

     \item The action $\Gamma \act X$ is amenable. 

     \item The Banach algebra $W_0(\Gamma,X)$ is left $\pi$--amenable. 

     \item The Banach algebra $\ker\bar{\pi}$ has a bounded right approximate identity.

     \item The Banach algebra $\ker\bar{\pi}$ is left $0$--amenable. 

     \end{enumerate}

     (We note that amenability of $W_0(\Gamma,X)$ is not equivalent to these.)
     \end{theorem}

     \begin{proof}

     Since $W_0$ is unital and $\ker\pi=\ker\bar{\pi}$, Proposition~4.3.7 of \cite{Runde20}
     gives $(2)\Leftrightarrow(3)$. Proposition~4.3.4(i) of \cite{Runde20} gives $(3)\Leftrightarrow(4)$.

     \vspace{2mm}
     (2)$\Rightarrow$(1) is shown by constructing an explicit derivation on $W_0$ as follows: 

     \begin{itemize}

     \item Take $E := (\ker\bar{\pi})^*$. Then  
     \[
          (\ker\bar{\pi})^{**} \cong \{ \tau \in W_0 ^{**} \mid \tau(\pi) =0 \} \; .
     \]

     \item The right action $(\ker\bar{\pi})^* \tca W_0$ is given by 
     \[
          \langle \Phi.  f_1  , f_2 \rangle := \langle \Phi , f_1 \ast f_2 \rangle  
     \]
     for $\Phi \in (\ker\bar{\pi})^*$, $f_1 \in W_0$, $f_2 \in \ker\bar{\pi}$. 

     \item The left action $W_0 \act (\ker\bar{\pi})^* $ is defined by 
     \[
     f .\Phi := \pi(f) \cdot \Phi \; .
     \] 

     \item Similarly, $W_0^*$ is a $W_0$--$W_0$--bimodule with the same structure.

     \item Take $\tau_0 \in W_0 ^{**}$ such that $\tau_0(\pi) =1$. 
     Then for each $f$, the element $f.\tau_0 -\tau_0 .f \in W_0^{**} $ actually lies in $(\ker\bar{\pi})^{**}$. 
     Indeed, 

     \begin{align*}
          (f.\tau_0 -\tau_0 .f) (\pi) 
          &= \tau_0(\pi .f - f.\pi) \\
          &= \tau_0 [W_0 \ni f' \mapsto \pi(f\ast f') - \pi(f) \pi(f') ] \\
          &= \tau_0 (0) = 0 \; .
     \end{align*}

     \item Define $D:W_0 \longrightarrow (\ker\bar{\pi})^{**}$ by
     \[
          D(f) := f.\tau_0 - \tau_0 .f 
     \]
     and this is a bounded derivation on $W_0$. 

     \item Hence by (2) there exists $\tau_1 \in (\ker\bar{\pi})^{**}$ 
     such that 
     \[
     f.\tau_0 - \tau_0 .f = f.\tau_1 - \tau_1 .f \; .
     \] 

     \item This means  
     \begin{align*}
       f.(\tau_0 - \tau_1) 
       &= (\tau_0 - \tau_1 ).f \\
       &= \pi(f) (\tau_0 - \tau_1 )   \; .
     \end{align*}

     In particular $\tau_0 - \tau_1 \in W_0^{**}$ is $\Gamma$--invariant 
     and $(\tau_0 - \tau_1)(\pi) = 1$. 
     Thus it is an invariant mean for $\Gamma \act X$. 
     \end{itemize}

     It remains to prove $(1)\Rightarrow(3)$. Let $(f_i)_i$ be an approximate $\Gamma$-invariant mean
     with $\pi(f_i)=\|f_i\|_{\mathcal A}=1$, and put $e_i=\delta_e-f_i$.
     Then $e_i\in\ker\bar{\pi}$ and $\|e_i\|_{\mathcal A}\leq 2$.
     For $a=p(\delta_g-\delta_e)$, where $p\in C(X)$ and $g\in\Gamma$, we have
     \[
     \|a-a\ast e_i\|_{\mathcal A}
     =\|p.(g.f_i-f_i)\|_{\mathcal A}
     \leq \|p\|_\infty\,\|g.f_i-f_i\|_{\mathcal A}
     \longrightarrow 0.
     \]
     The linear span of these elements is dense in $\ker\bar{\pi}$.
     Since $(e_i)_i$ is bounded, $a\ast e_i\to a$ for every $a\in\ker\bar{\pi}$.
     Thus $(e_i)_i$ is a bounded right approximate identity.
     \end{proof}

\section{Application to Exact Groups}\label{sec:application_to_exact_groups}

     In this section, we investigate the canonical action $\Gamma \act \beta\Gamma$ as a special case. 
     Recall $C(\beta\Gamma) \cong \ell^\infty(\Gamma)$. 
     With complex coefficients, the construction of Douglas and Nowak~\cite[arXiv version, Definition~2.1]{Dougl10} gives the uniform convolution algebra $\ell_u\Gamma=\mathcal{A}(\Gamma\act\beta\Gamma)$.

     \begin{corollary} \label{cor:exact group}
     For a discrete group $\Gamma$, the following are equivalent: 
     \begin{enumerate}
          \item The group $\Gamma$ is exact. 
          \item There exists a nonempty compact Hausdorff $\Gamma$--space $X$ 
          such that $\mathcal{A}(\Gamma\act X)$ is right--$C(X)$--$\ell^1$--amenable. 
          \item The Banach algebra $\mathcal{A}(\Gamma\act\beta\Gamma)$ is right--$\ell^\infty(\Gamma)$--$\ell^1$--amenable. 
     \end{enumerate}
     \end{corollary}

     \begin{proof}
     By \autoref{thm:main} and $C(\beta\Gamma)=\ell^\infty(\Gamma)$, conditions (2) and (3)
     are equivalent to the existence of an amenable action of $\Gamma$ on a nonempty compact Hausdorff space
     and to amenability of $\Gamma\act\beta\Gamma$, respectively.
     These two action conditions are equivalent by Anantharaman-Delaroche~\cite[published version, Proposition~3.4]{AnantharamanDelaroche02},
     and the latter is equivalent to (1) by Ozawa~\cite{Ozawa00}.
     \end{proof}

\section{Further Directions}\label{sec:to_further_results}

     It is natural to ask whether \autoref{thm:main} extends to actions of topological groups,
     to actions on $C^*$-algebras, or to topological groupoids.

     \subsection{Characterization by Approximate Diagonals}

     A Banach algebra is amenable if and only if it has a bounded approximate diagonal,
     or equivalently a virtual diagonal \cite[Theorem~2.2.5]{Runde20}.
     In the proof that amenability gives a virtual diagonal, one uses the innerness of a derivation into a dual bimodule.
     In our setting, \autoref{def:right-continuous-functions-ell1} gives innerness only when the predual
     is right--$C(X)$--$\ell^1$--geometric.
     Can amenability of $\Gamma\act X$ be characterized by an analogue of a bounded approximate diagonal for $\mathcal A$?

\section{Acknowledgments}

     The author is grateful to Narutaka Ozawa for his invaluable advice throughout this research and for teaching the proof techniques used in Section 3 and for suggesting the earlier approach via measured groupoids.
     The author would also like to thank his advisor, Yasuyuki Kawahigashi, for his support and guidance in various aspects of this work. 
     In addition, the author appreciates the stimulating discussions with colleagues Ikhan Choi, Ayoub Hafid, and Miho Mukohara, as well as his friend Ken Sato---all of which greatly contributed to this research.
     The author used GPT in preparing this manuscript.

\end{document}